\documentclass[11pt, a4paper, notitlepage]{article}

\usepackage{amsmath, latexsym, amsthm, amsfonts, amssymb} 
\usepackage{natbib} 

\bibpunct{[}{]}{,}{a}{}{,} 

\theoremstyle{plain}
\newtheorem{thm}{Theorem}[section]

\newtheorem{lem}{Lemma}[section]

\newtheorem{cnd}{Condition}[section]

\theoremstyle{definition} 

\newcommand{\infint}{\int_{-\infty}^{\infty}}
\newcommand{\Log}{\operatorname{Log}}

\def\ex{{\rm E\,}}

\begin{document}
\date{\today}
\title{Nonparametric estimation of the characteristic triplet of a discretely observed L\'evy process}
\author{Shota Gugushvili\\
{\normalsize Eurandom}\\
{\normalsize Technische Universiteit Eindhoven}\\
{\normalsize P.O. Box 513}\\
{\normalsize 5600 MB Eindhoven}\\
{\normalsize The Netherlands}\\
{\normalsize gugushvili@eurandom.tue.nl}}
\maketitle
\begin{abstract}
Given a discrete time sample $X_1,\ldots X_n$ from a L\'evy
process $X=(X_t)_{t\geq 0}$ of a finite jump activity, we study
the problem of nonparametric estimation of the characteristic
triplet $(\gamma,\sigma^2,\rho)$ corresponding to the process $X.$
Based on Fourier inversion and kernel smoothing, we propose
estimators of $\gamma,\sigma^2$ and $\rho$ and study their
asymptotic behaviour. The obtained results include derivation of
upper bounds on the mean square error of the estimators of
$\gamma$ and $\sigma^2$ and an upper bound on the mean integrated
square error of an estimator of $\rho.$
\end{abstract}
\newpage

\section{Introduction}

L\'evy processes are stochastic processes with stationary
independent increments. The class of such processes is extremely
rich, the best known representatives being Poisson and compound
Poisson processes, Brownian motion, Cauchy process and, more
generally, stable processes. Though the basic properties of L\'evy
processes have been well-studied and understood since a long time,
see e.g.\ \cite{skorohod}, during the last years there has been a
renaissance of interest in L\'evy processes. This revival of
interest is mainly due to the fact that L\'evy processes found
numerous applications in practice and proved to be useful in a
broad range of fields, including finance, insurance, queueing,
telecommunications, quantum theory, extreme value theory and many
others, see e.g.\ \cite{barndorff3} for an overview. \cite{tankov}
provides a thorough treatment of applications of L\'evy processes
in finance. Comprehensive modern texts on fundamentals of L\'evy
processes are \cite{bertoin,kyprianou,sato}, and we refer to those
for precise definitions and more details concerning properties of
L\'evy processes.

Already from the outset an intimate relation of L\'evy processes
with infinitely divisible distributions was discovered. For a
detailed exposition of infinitely divisible distributions see
e.g.\ \cite{steutel}. In fact there is a one-to-one correspondence
between L\'evy processes and infinitely divisible distributions:
if $X=(X_t)_{t\geq 0}$ is a L\'evy process, then its marginal
distributions are all infinitely divisible and are determined by
the distribution of $X_1.$ Conversely, given an infinitely
divisible distribution $\mu,$ one can construct a L\'evy process,
such that $P_{X_1}=\mu.$ The celebrated L\'evy-Khintchine formula
for infinitely divisible distributions provides us with an
expression for the characteristic function of $X_1,$ which can be
written as
\begin{equation}
\label{levykhintchineformula} \phi_{X_1}(z)=\operatorname{exp}\left[i\gamma
z-\frac{1}{2}\sigma^2z^2+\int_{\mathbb{R}}(e^{izx}-1-izx 1_{[-1,1]}(x))\nu(dx)\right],
\end{equation}
where $\gamma\in\mathbb{R},\sigma\geq 0$ and $\nu$ is a measure
concentrated on $\mathbb{R}\backslash\{0\},$ such that
$\int_{\mathbb{R}}(1\wedge x^2)\nu(dx)<\infty.$ This measure is
called the L\'evy measure corresponding to the L\'evy process $X,$
while the triple $(\gamma,\sigma^2,\nu)$ is referred to as the
characteristic or L\'evy triplet of $X.$ The representation in
\eqref{levykhintchineformula} in terms of the triplet
$(\gamma,\sigma^2,\nu)$ is unique. Thus the L\'evy triplet
provides us with means for unique characterisation of a law of any
L\'evy process. Bearing this in mind, the statistical inference
for L\'evy processes can be reduced to inference on the
characteristic triplet. There are several ways to approach
estimation problems for L\'evy processes: parametric,
nonparametric and semiparametric approaches. These approaches
depend on whether one decides to parametrise the L\'evy measure
(or its density, in case it exists) with a Euclidean parameter, or
to work in a nonparametric setting. A semiparametric approach to
parametrisation of the L\'evy measure is also possible. Most of
the existing literature dealing with estimation problems for
L\'evy processes is concerned with parametric estimation of the
L\'evy measure (or its density, in case it exists), see e.g.\
\cite{akritas2,akritas1}, where a fairly general
setting is considered. There are relatively few papers that study
nonparametric inference procedures for L\'evy processes, and the
majority of them assume that high frequency data are available,
i.e.\ either a L\'evy process is observed continuously over a time
interval $[0,T]$ with $T\rightarrow\infty,$ or it is observed at
equidistant time points $\Delta_n,\ldots,n\Delta_n$ and
$\lim_{n\rightarrow \infty}\Delta_n=0,$
$\lim_{n\rightarrow\infty}n\Delta_n=\infty,$ see e.g.\
\cite{basawa3,figueroa,rubin}. On the other hand it is equally
interesting to study estimation problems for the case when the
high frequency data are not available, i.e.\ when
$\Delta_n=\Delta$ is kept fixed. The latter case is more involved
due to the fact that the information on the L\'evy measure is
contained in jumps of the process $X$ and impossibility to observe
them directly as in the case of a continuous record of
observations, or to `disentangle' them from the Brownian motion as
in the high frequency data setting, makes the estimation problem
rather difficult. In the particular context of a compound Poisson
process we mention \cite{bu,bugr,gug}, where given a sample
$Y_1,\ldots,Y_n$ from a compound Poisson process $Y=(Y_t)_{t\geq
0},$ nonparametric estimators of the jump size distribution
function $F$ (see \cite{bu,bugr}) and its density $f$ (see
\cite{gug}) were proposed and their asymptotics were studied as
$n\rightarrow \infty.$ This problem is referred to as
decompounding. Nonparametric estimation of the L\'evy measure
$\nu$ based on low frequency observations from a general L\'evy
process $X$ was studied in \cite{reiss,watteel}.
However, these papers treat the case of estimation of the L\'evy
measure only (or of the canonical function $K$ in case of
\cite{watteel}) and not of its density. Moreover, the
rates of convergence of the proposed estimators are studied under the strong
moment condition $\ex[|X_1|^{4+\delta}]<\infty,$ where $\delta$ is
some strictly positive number. This condition automatically
excludes distributions with heavy tails. Nonparametric estimation
of the L\'evy density of a pure jump L\'evy process (i.e. a L\'evy
process without a drift and a Brownian component) was considered
in \cite{comte}. We refer to those papers for additional details.

In the present work we concentrate on nonparametric inference for
L\'evy processes that are of finite jump activity and have
absolutely continuous L\'evy measures. In essence this means that
we consider a superposition of a compound Poisson process and an
independent Brownian motion. The L\'evy-Khintchine formula in our
case takes the form
\begin{equation}
\label{levykhintchineformula2}
\phi_{X_1}(z)=\operatorname{exp}\left[i\gamma z
-\frac{1}{2}\sigma^2z^2+\int_{\mathbb{R}}(e^{izx}-1)\rho(x)dx\right],
\end{equation}
where the L\'evy density $\rho$ is such that $\lambda:=\infint
\rho(x)dx<\infty.$ To keep the notation compact, we again use
$\gamma$ to denote the drift coefficient in
\eqref{levykhintchineformula2}, even though it is in general
different from $\gamma$ in \eqref{levykhintchineformula}. Observe
that the process $X$ is related to Merton's jump-diffusion model
of an asset price, see \cite{merton}. Additional details on
exponential L\'evy models, of which Merton's model is a particular
case, can be found e.g.\ in \cite{tankov}.

Suppose that we dispose a sample
$X_{\Delta},X_{2\Delta},\ldots,X_{n\Delta}$ from the process $X.$
By a rescaling argument, without loss of generality, we may take
$\Delta=1.$ Based on this sample, our goal is to infer the
characteristic triplet $(\gamma,\sigma^2,\rho),$ corresponding to
\eqref{levykhintchineformula2}. At this point we mention that a
problem related to ours was studied in \cite{belomestny}. There an
exponential of the process $X$ (this exponential models evolution
of an asset price over time) was considered and inference was
drawn on parameters $\sigma,\lambda$ and $\gamma$ and and the
functional parameter, the L\'evy density $\rho,$ based on
observations on prices of vanilla options on this asset. The
difference of our estimation problem with this problem is the
observation scheme, since we observe directly the process $X.$
Moreover, existence of an exponential moment of $X$ was assumed in
\cite{belomestny} (this is unavoidable in the financial setting,
because otherwise one cannot price financial derivatives).

Our estimators of $\gamma,\lambda$ and $\sigma^2$ will be based on
\eqref{levykhintchineformula2} and the use of a plug-in device. To
estimate $\rho,$ we will use methods developed in nonparametric
density estimation based on i.i.d.\ observations, in particular we
will employ the Fourier inversion approach and kernel smoothing,
see e.g.\ Sections 6.3 and 10.1 in \cite{wasserman} for an
overview. In fact by the stationary independent increments
property of a L\'evy process, see Definition 1.6 in \cite{sato},
the problem of estimating $(\gamma,\sigma^2,\rho)$ from a discrete
time sample $X_1,\ldots,X_n$ from the process $X$ is equivalent to
the following one (to keep the notation compact, we again use
$X$'s to denote our observations): let $X_1,\ldots, X_n$ be
i.i.d.\ copies of a random variable $X$ with characteristic
function given by \eqref{levykhintchineformula2} (in the sequel we
will use $X$ to denote a generic observation). Based on these
observations, the problem is to construct estimators of
$\gamma,\sigma^2$ and $\rho.$ We henceforth will concentrate on
this equivalent problem.

The rest of the paper is organised as follows: in Section
\ref{sigmalambda} we construct consistent estimators of parameters
$\sigma^2,\lambda$ and $\gamma.$ In Section
\ref{deconvnoise-results}, using the estimators of
$\sigma^2,\lambda$ and $\gamma,$ we propose a plug-in type
estimator for $\rho$ and study the behaviour of its mean
integrated square error. In Section \ref{lowbound} we derive a lower bound for estimation of $\rho.$ All the proofs are collected in Section
\ref{deconvnoise-proofs}.

\section{Estimation of $\sigma,\lambda$ and $\gamma$}
\label{sigmalambda}

In the sequel we will find it convenient to use the jump size
density $f(x):=\rho(x)/\lambda.$ We first formulate conditions on
$\rho,\sigma$ and $\gamma,$ that will be used throughout the paper.

\begin{cnd}
\label{conditionf} Let the unknown density $\rho$ belong to the class
\begin{align*}
W(\beta,L,\Lambda,K)=\Bigl\{ &\rho:\rho(x)=\lambda
f(x),f\text{ is a density},\infint x^2f(x)dx\leq K,\\
&\infint |t|^{\beta}|\phi_f(t)|dt\leq L,
\lambda\in(0,\Lambda]\Bigr\},
\end{align*}
where $\beta,L,\Lambda$ and $K$ are strictly positive numbers.
\end{cnd}
This condition implies in particular that the Fourier transform
$\phi_{\rho}(t)=\lambda\phi_f(t)$ of $\rho$ is integrable. The
latter is natural in light of the fact that our estimation
procedure for $\rho$ will be based on Fourier inversion, see
Section \ref{deconvnoise-results}. The integrability of
$\phi_{\rho}$ implies that $\rho$ is bounded and continuous. It
follows that $f$ is bounded and continuous, and hence, being a
probability density, it is also square integrable. Therefore
$\rho(x)=\lambda f(x)$ is square integrable as well. This again is
a natural assumption, because we will select the mean integrated
square error as a performance criterion for our estimator of
$\rho.$ The condition $\lambda>0$ ensures that the process $X$ has
a compound Poisson component. Restriction of the class of
densities $f$ to those densities that have the finite second
moment is needed to ensure that $\ex[X^2]$ is bounded from above
uniformly in $\rho,\gamma$ and $\sigma.$ The latter is a technical
condition used in the proofs.

\begin{cnd}
\label{conditionsigma} Let $\sigma$ be such that
$\sigma\in(0,\Sigma],$ where $\Sigma$ is a strictly
positive number.
\end{cnd}
This is not a restrictive assumption in many applications, since
for instance in the financial context $\sigma,$ which models
volatility, typically belongs to some bounded set, e.g.\ a compact
$[0,\Sigma]$ as in \cite{belomestny}. The condition $\sigma>0$ in
our case ensures that $X$ has a Brownian component. If $\sigma=0,$
then our problem in essence reduces to the one studied in
\cite{gug}.

\begin{cnd}
\label{conditiongamma} Let $\gamma$ be such that $|\gamma|\leq\Gamma,$
where $\Gamma$ denotes a positive number.
\end{cnd}
Remarks similar to those we made after Condition \ref{conditionsigma} apply in this case as well.

Next we turn to the construction of estimators of
$\sigma^2,\lambda$ and $\gamma.$ The ideas we use resemble those
in \cite{belomestny}. Let $\Re(z)$ and $\Im(z)$ denote the real
and the imaginary parts of a complex number $z,$ respectively.
From \eqref{levykhintchineformula2} we have
\begin{equation}
\label{decomposition} \log
\left(|\phi_X(t)|\right)=-\lambda+\lambda\Re(\phi_f(t))-\frac{\sigma^2t^2}{2}.
\end{equation}
Here we used the fact that
\begin{equation*}
\log\left( \left|e^{\lambda\phi_f(t)}\right|\right)=\log
\left(e^{\lambda\Re(\phi_f(t))}\right)+\log
\left(\left|e^{i\lambda\Im(\phi_f(t))}\right|\right)
=\lambda\Re(\phi_f(t)).
\end{equation*}
Let $v^h$ be a kernel that depends on a bandwidth $h$ and is such
that
\begin{equation*}
\int_{-1/h}^{1/h}v^h(t)dt=0, \quad \int_{-1/h}^{1/h}\left(-\frac{t^2}{2}\right)v^h(t)dt=1.
\end{equation*}
Observe that unlike kernels in kernel density estimation, see e.g.\ Definition 1.3 in \cite{tsyb}, the
function $v^{h}$ does not integrate to one and by calling it a
kernel we abuse the terminology. In view of \eqref{decomposition}
\begin{equation}
\label{sig1} \int_{-1/h}^{1/h}\log
(|\phi_X(t)|)v^h(t)dt=\lambda\int_{-1/h}^{1/h}\Re(\phi_f(t))v^h(t)dt+\sigma^2.
\end{equation}
Provided enough assumptions on $v^h,$ one can achieve that the
right-hand side of \eqref{sig1} tends to $\sigma^2$ as
$h\rightarrow 0.$ A natural way to construct an estimator of
$\sigma^2$ then is to replace in \eqref{sig1} $\log (|\phi_X(t)|)$
by its estimator $\log (|\phi_{emp}(t)|).$ Consequently, we
propose
\begin{equation}
\label{sig2}
\tilde{\sigma}_n^2=\int_{-1/h}^{1/h}\max\{\min\{M_n,\log
(|\phi_{emp}(t)|)\},-M_n\}v^h(t)dt
\end{equation}
as an estimator of $\sigma^2.$ Here $M_n$ denotes a sequence of
positive numbers diverging to infinity at a suitable rate. The
truncation in \eqref{sig2} is introduced due to technical reasons
in order to obtain a consistent estimator.

We now state our assumptions on the kernel $v_h,$ the bandwidth
$h$ and the sequence $M=(M_n)_{n\geq 1}.$

\begin{cnd}
\label{conditionv}
Let the kernel $v^h(t)=h^3v(ht),$ where the function $v$ is continuous and real-valued, has a support on $[-1,1]$ and is such that
\begin{equation*}
\int_{-1}^1v(t)dt=0, \quad \int_{-1}^1\left(-\frac{t^2}{2}\right)v(t)dt=1, \quad v(t)=O(t^{\beta}) \text{ as } t\rightarrow 0.
\end{equation*}
Here $\beta$ is the same as in Condition \ref{conditionf}.
\end{cnd}
\begin{cnd}
\label{conditionh} Let the bandwidth $h$ depend on $n$ and be
such that $h_n=(\eta\log n)^{-1/2}$ with $0<\eta<\Sigma^{-2}.$
\end{cnd}
Using a default convention in kernel density estimation, we will
suppress the index $n$ when writing $h_n,$ since no ambiguity will
arise. Condition \ref{conditionh} implies that
$ne^{-\Sigma^2/h^2}\rightarrow\infty,$ since the logarithm of the
left-hand side of this expression diverges to minus infinity.
Condition \ref{conditionh} is required to establish consistency of
estimators of $\sigma^2,\lambda,\gamma$ and $\rho.$ Hence it is of
the asymptotic nature. For finite samples of moderate size,
however, it might lead to unsatisfactory estimates. A separate
simulation study in the spirit of \cite{delaigle} is needed to
study possible bandwidth selection methods in practical problems.
\begin{cnd}
\label{conditionm} Let the truncating sequence $M=(M_n)_{n\geq 1}$
be such that $M_n=m_n h^{-2},$ where $m_n$ is a sequence of real
numbers diverging to plus infinity at a slower rate than $\log n,$
for instance $m_n=\log\log n.$
\end{cnd}
Other restrictions on $M$ are also possible.

In the sequel we will frequently employ the symbol $\lesssim$ and
$\gtrsim,$ meaning `less or equal up to a universal constant', or
`greater or equal up to a universal constant', respectively. The
following theorem establishes consistency of $\tilde{\sigma}_n^2.$
\begin{thm}
\label{thm-sigmatilde} Let Conditions \ref{conditionf}--\ref{conditionm} be satisfied and let the
estimator $\tilde{\sigma}_n^2$ be defined by \eqref{sig2}. Then
\begin{equation*}
\sup_{|\gamma|\leq \Gamma}\sup_{\sigma\in(0,\Sigma]}\sup_{\rho\in
W(\beta,L,\Lambda,K)}\ex[(\tilde{\sigma}_n^2-\sigma^2)^2]\lesssim(\log
n)^{-\beta-3}.
\end{equation*}
\end{thm}

To construct an estimator of the jump intensity $\lambda,$ we will
again use \eqref{levykhintchineformula2}, but now in a different
way. Let $u^h$ denote a kernel that depends on $h$ and is such
that
\begin{equation*}
\int_{-1/h}^{1/h}u^h(t)dt=-1, \quad \int_{-1/h}^{1/h}{t^2}u^h(t)dt=0.
\end{equation*}
Then
\begin{equation}
\label{lam1} \int_{-1/h}^{1/h}\log
(|\phi_X(t)|)u^h(t)dt=\lambda+\lambda\int_{-1/h}^{1/h}\Re(\phi_f(t))u^h(t)dt.
\end{equation}
With a proper selection of $u^h$ one can ensure that \eqref{lam1}
converges to $\lambda$ as $h\rightarrow 0.$ Using a plug-in
device, we therefore propose the following estimator of $\lambda$:
\begin{equation*}
\tilde{\lambda}_n=\int_{-1/h}^{1/h}\max\{\min\{M_n,\log
(|\phi_{emp}(t)|)\},-M_n\}u^h(t)dt.
\end{equation*}
Now we state a condition on the kernel $u^h.$
\begin{cnd}
\label{conditionu}
Let the kernel $u^h(t)=hu(ht),$ where the function $u$ is continuous and real-valued, has a support on $[-1,1]$ and is such that
\begin{equation*}
\int_{-1}^1u(t)dt=-1, \quad \int_{-1}^1{t^2}u(t)dt=0, \quad u(t)=O(t^{\beta}) \text{ as } t\rightarrow 0.
\end{equation*}
Here $\beta$ is the same as in Condition \ref{conditionf}.
\end{cnd}
The following theorem deals with asymptotics of the estimator
$\tilde{\lambda}_n.$
\begin{thm}
\label{thm-lambdatilde} Let Conditions
\ref{conditionf}--\ref{conditiongamma} and
\ref{conditionh}--\ref{conditionu} be satisfied and let the
estimator $\tilde{\lambda}_n$ be defined by \eqref{lam1}. Then
\begin{equation*}
\sup_{|\gamma|\leq\Gamma}\sup_{\sigma\in(0,\Sigma]}\sup_{\rho\in
W(\beta,L,\Lambda,K)}\ex[(\tilde{\lambda}_n-\lambda)^2]\lesssim(\log
n)^{-\beta-1}.
\end{equation*}
\end{thm}

Finally, we consider estimation of the drift coefficient $\gamma.$ By
\eqref{levykhintchineformula2} we have
\begin{equation*}
\Im(\Log(\phi_X(t)))=\gamma t+\lambda\Im(\phi_f(t)),
\end{equation*}
where $\Log(\phi_X(t))$ denotes the distinguished logarithm of the
characteristic function $\phi_X(t),$ i.e.\ a logarithm that is a
single-valued and continuous function of $t,$ such that
$\Log(\phi_X(0))=0,$ see Theorem~7.6.2 in \cite{chung} for details
of its construction. Let $w^h$ denote a kernel that depends on $h$
and is such that
\begin{equation*}
\int_{-1/h}^{1/h}tw^h(t)dt=1.
\end{equation*}
Then
\begin{equation*}
\int_{-1/h}^{1/h}\Im(\Log(\phi_X(t)))w^{h}(t)dt=\gamma+\lambda\int_{-1/h}^{1/h}\Im(\phi_f(t))w^{h}(t)dt.
\end{equation*}
With an appropriate choice of $w^{h}$ the right-hand side will
converge to $\gamma.$ Therefore, by a plug-in device, for those
$\omega$'s from the underlying sample space $\Omega$ for which the
distinguished logarithm can be defined, we define an estimator of
$\gamma$ as
\begin{equation}
\label{gam1}
\tilde{\gamma}_n=\int_{-1/h}^{1/h}\max\{\min\{\Im(\Log(\phi_{emp}(t))),M_n\},-M_n\}w^{h}(t)dt,
\end{equation}
while for those $\omega$'s for which it cannot be defined, we
assign an arbitrary value to the distinguished logarithm in
\eqref{gam1}, e.g. zero. The distinguished logarithm in
\eqref{gam1} can be defined only for those $\omega $'s for which
$\phi_{emp}(t)$ as a function of $t$ does not vanish on
$[-h^{-1},h^{-1}],$ see Theorem~7.6.2 in \cite{chung}. In fact the
probability of the exceptional set, where the distinguished
logarithm is undefined, tends to zero  as $n\rightarrow\infty.$ We
will show this by finding a set $B_{n},$ such that on this set the
distinguished logarithm might be undefined, while on its
complement $B_{n}^c$ it is necessarily well-defined. We have
\begin{equation}
\label{lb} \inf_{t\in[-h^{-1},h^{-1}]}|\phi_X(t)|\geq
e^{-2\lambda-\sigma^2/(2h^2)}\geq e^{-2\Lambda-\Sigma^2/(2h^2)}.
\end{equation}
Define
\begin{equation}
\label{bndef}
\begin{split}
B_{n}&=\left\{\sup_{t\in[-h^{-1},h^{-1}]}\left|{\phi_{emp}(t)}-{\phi_{X}(t)}\right|>\delta\right\},\\
B_{n}^c&=\left\{\sup_{t\in[-h^{-1},h^{-1}]}\left|{\phi_{emp}(t)}-{\phi_{X}(t)}\right|\leq\delta\right\},
\end{split}
\end{equation}
with $\delta=(1/2)e^{-2\Lambda-\Sigma^2/(2h^2)}.$ From \eqref{lb},
\eqref{bndef} and Theorem 7.6.2 of \cite{chung} it follows that on
the set $B_{n}^c$ the distinguished logarithm is well-defined
(with $t$ restricted to $[-h^{-1},h^{-1}]$), since on this set
$\phi_{emp}$ cannot take the value zero. Notice that given our
conditions on $\rho$ and $\sigma$, our choice of $\delta$ is
suitable whatever $\rho,\sigma$ and $\gamma$ are. All we need to
show is that $\operatorname{P}(B_{n})\rightarrow 0.$ The following
theorem holds true.
\begin{thm}
\label{phideviation} Let $W(\beta,L,\Lambda,K)$ be defined as in
Condition \ref{conditionf}. Then
\begin{equation*}
\sup_{|\gamma|\leq\Gamma}\sup_{\sigma\in(0,\Sigma]}\sup_{\rho\in W(\beta,L,\Lambda,K)}\operatorname{P}(B_{n})\lesssim \frac{e^{\Sigma^2/h^2}}{nh^2}.
\end{equation*}
\end{thm}
Notice that by Condition \ref{conditionh} we have $\operatorname{P}(B_{n})\rightarrow 0.$ We now state a condition on the kernel $w^h.$
\begin{cnd}
\label{cndwh} Let the kernel $w^h(t)=h^2w(ht),$ where the function
$w$ is continuous and real-valued, has a support on $[-1,1]$ and
is such that
\begin{equation*}
\int_{-1}^1tw(t)dt=1, \quad w(t)=O(t^{\beta}) \text{ as }
t\rightarrow 0.
\end{equation*}
Here $\beta$ is the same as in Condition \ref{conditionf}.
\end{cnd}

The following result holds.

\begin{thm}
\label{thm-gamma} Let Conditions
\ref{conditionf}--\ref{conditiongamma},
\ref{conditionh}--\ref{conditionm} and \ref{cndwh} be
satisfied and let the estimator $\tilde{\gamma}_n$ be defined by
\eqref{gam1}. Then
\begin{equation*}
\sup_{|\gamma|\leq\Gamma}\sup_{\sigma\in(0,\Sigma]}\sup_{\rho\in
W_{sym}(\beta,L,\Lambda,K)}\ex[(\tilde{\gamma}_n-\gamma)^2]\lesssim(\log
n)^{-\beta-2},
\end{equation*}
where $W_{sym}(\beta,L,\Lambda,K)$ denotes the class of symmetric
L\'evy densities that belong to $W(\beta,L,\Lambda,K).$
\end{thm}
The reason why we restrict ourselves to the class of symmetric
L\'evy densities is that we would like to obtain a uniformly
consistent estimator of $\gamma$ (and eventually of $\rho,$ see
Section \ref{deconvnoise-results}). The main technical difficulty
in this respect is the (uniform) control of the argument (i.e.\ of
the imaginary part) of the distinguished logarithm in
\eqref{gam1}, see the proofs of Theorems \ref{thm-gamma} and
\ref{thm-gam2}. For transparency purposes we restrict ourselves to
the class of symmetric $\rho$'s.
If we are only interested in the consistency of the
estimator for a fixed $\rho,$ then the above restriction is not
needed and the result holds without it. We formulate the
corresponding theorem below.
\begin{thm}
\label{thm-gam2} Let Conditions \ref{conditionh}--\ref{conditionm}
and \ref{cndwh} be satisfied. Furthermore, let
$\gamma\in\mathbb{R},\sigma^2>0$ and let $\rho$ be such that
\begin{equation}
\label{rhocond} 0<\lambda<\infty; \quad \infint x^2f(x)dx<\infty;
\quad \infint |t|^{\beta}|\phi_{f}(t)|dt<\infty.
\end{equation}
Let the estimator $\tilde{\gamma}_n$ be defined by
\eqref{gam1}. Then
\begin{equation*}
\ex[(\tilde{\gamma}_n-\gamma)^2]\lesssim(\log
n)^{-\beta-2}.
\end{equation*}
\end{thm}

Now that we obtained uniformly consistent estimators of
$\sigma^2,\lambda$ and $\gamma,$ we can move to the construction of an
estimator of $\rho.$

\section{Estimation of $\rho$}
\label{deconvnoise-results}

The method that will be used to construct an estimator of $\rho$
is based on Fourier inversion and is similar to the approach in
\cite{gug}. Solving for $\phi_{\rho}$ in
\eqref{levykhintchineformula2}, we get
\begin{equation}
\label{distlog}
\phi_{\rho}(t)=\operatorname{Log}\left(\frac{\phi_X(t)}{e^{i\gamma
t }e^{-\lambda}e^{-\sigma^2t^2/2}}\right).
\end{equation}
Here $\operatorname{Log}$ again denotes the distinguished
logarithm, which can be constructed as in Theorem 7.6.2 of
\cite{chung} taking into account an obvious difference that in our
case the function $e^{\phi_{\rho}(t)}$ equals $e^{\lambda}$ at
$t=0.$

By Fourier inversion we have
\begin{equation*}
\rho(x)=\frac{1}{2\pi}\infint
e^{-itx}\operatorname{Log}\left(\frac{\phi_X(t)}{e^{i\gamma t
}e^{-\lambda}e^{-\sigma^2t^2/2}}\right)dt.
\end{equation*}
This expression will be used as the basis for construction of an
estimator of $\rho.$ Let $k$ be a symmetric kernel with Fourier
transform $\phi_k$ supported on $[-1,1]$ and nonzero there, and
let $h>0$ be a bandwidth. Since the characteristic function
$\phi_X$ is integrable, there exists a density $q$ of $X,$ and
moreover, it is continuous and bounded. This density can be
estimated by a kernel density estimator
\begin{equation*}
q_{n}(x)=\frac{1}{nh}\sum_{j=1}^n k\left(\frac{x-X_j}{h}\right),
\end{equation*}
see e.g.\ \cite{tsyb,wasserman} for an introduction to kernel
density estimation. Its characteristic function
$\phi_{emp}(t)\phi_k(ht)$ will then serve as an estimator of
$\phi_X(t).$ For those $\omega$'s from the sample space $\Omega,$
for which the distinguished logarithm in the integral below is
well-defined, $\rho$ can be estimated by the plug-in type
estimator,
\begin{equation}
\label{deconvnoise-fnh}
\rho_{n}(x)=\frac{1}{2\pi}\int_{-1/h}^{1/h}
e^{-itx}\operatorname{Log}\left(\frac{\phi_{emp}(t)\phi_k(ht)}{e^{i\tilde{\gamma}_nt}e^{-\tilde{\lambda}_n}e^{-\tilde{\sigma}^2_nt^2/2}}\right)dt,
\end{equation}
while for those $\omega$'s, for which the distinguished logarithm
cannot be defined, we can assign an arbitrary value to
$\rho_{n}(x),$ e.g.\ zero. Notice that the estimator
\eqref{deconvnoise-fnh} is real-valued, which can be seen by
changing the integration variable from $t$ into $-t.$

Our definition of the estimator is quite intuitive, however in
order to investigate its asymptotic behaviour, some modifications
are due: we need to introduce truncation in the definition of $\rho_n$ and consequently, we propose
\begin{equation}
\label{deconvnoise-hatfnh}
\begin{split}
\hat{\rho}_{n}(x)&=-i \tilde{\gamma}_n\frac{1}{2\pi}\int_{-1/h}^{1/h}e^{-itx}tdt+ \tilde{\lambda}_n\frac{1}{2\pi}\int_{-1/h}^{1/h}e^{-itx}dt+\frac{\tilde{\sigma}_n^2}{2}\frac{1}{2\pi}\int_{-1/h}^{1/h}e^{-itx}t^2dt\\
&+\frac{1}{2\pi}\int_{-1/h}^{1/h}e^{-itx} \max\{\min\{M_n,\log(|\phi_{emp}(t)\phi_k(ht)|)\},-M_n\}dt\\
&+i\frac{1}{2\pi}\int_{-1/h}^{1/h}e^{-itx}\max\{\min\{M_n,\arg(\phi_{emp}(t)\phi_k(ht))\},-M_n\}dt
\end{split}
\end{equation}
as an estimator of $\rho(x).$ Here $M=(M_n)_{n\geq1}$ denotes a
sequence of positive numbers satisfying Condition
\ref{conditionm}, while $\log$ and $\arg$ are the real and
imaginary parts of the distinguished logarithm, respectively.
Notice that in \eqref{deconvnoise-hatfnh} we essentially truncate
the real and imaginary parts of the distinguished logarithm from
above and from below. The truncation is only necessary to make
asymptotic arguments work and in practice we do not need to employ
it. Observe that $|\hat{\rho}_{n}(x)|^2$ is integrable, since by
Parseval's identity each summand in \eqref{deconvnoise-hatfnh} is
square integrable. Furthermore, by Theorem \ref{phideviation} the
probability of the set, where the distinguished logarithm in
\eqref{deconvnoise-hatfnh} can be defined, tends to one as the
sample size $n$ tends to infinity.

We now state a condition on the kernel $k$ that will be used when
studying asymptotics of $\hat{\rho}_{n}.$

\begin{cnd}
\label{conditionk} Let the kernel $k$ be the sinc kernel,
$k(x)=\sin x/(\pi x).$
\end{cnd}
The Fourier transform of the sinc kernel is given by
$\phi_k(t)=1_{[-1,1]}(t).$ The use of the sinc kernel in our
problem is equivalent to the use of the spectral cut-off method in
\cite{belomestny} in a problem similar to ours. The sinc kernel
has been used successfully in kernel density estimation since a
long time, see e.g.\ \cite{davis1,davis2}. An attractive feature
of the sinc kernel in ordinary kernel density estimation is that
it is asymptotically optimal when one selects the mean square
error or the mean integrated square error as the criterion of the
performance of an estimator. Notice that the sinc kernel is not
Lebesgue integrable, but its square is.

Now we will study the
asymptotics of $\hat{\rho}_n.$ As a criterion of performance of
the estimator $\hat{\rho}_n$ we select the mean integrated square
error
\begin{equation*}
\operatorname{MISE}[\hat{\rho}_{n}]=\ex\left[\infint
|\hat{\rho}_{n}(x)-\rho(x)|^2dx\right].
\end{equation*}
Other possible choices include, for instance, the mean square
error and the mean integrated error of the estimator. These are
not discussed here. The theorem given below constitutes the main
result of the paper. It provides an order bound on
$\operatorname{MISE}[\hat{\rho}_{n}]$ over an appropriate class of
characteristic triplets and demonstrates that the estimator
$\hat{\rho}_n$ is consistent in the $\operatorname{MISE}$ sense.
\begin{thm}
\label{thm-f} Assume that assumptions of Theorems
\ref{thm-sigmatilde}--\ref{thm-gamma} hold. Let the
estimator $\hat{\rho}_n$ be defined by \eqref{deconvnoise-hatfnh}. Then
\begin{equation*}
\sup_{|\gamma|\leq\Gamma}\sup_{\sigma\in(0,\Sigma]}\sup_{\rho\in
W_{sym}^*(\beta,L,C,\Lambda,K)}\operatorname{MISE}[\hat{\rho}_{n}]\lesssim
(\log n)^{-\beta},
\end{equation*}
where $W_{sym}^*(\beta,L,C,\Lambda,K)$ denotes the class of L\'evy
densities $\rho,$ such that $\rho\in W_{sym}(\beta,L,\Lambda,K)$
and additionally \begin{equation*} \infint
|t|^{2\beta}|\phi_f(t)|^2dt\leq C.
\end{equation*}
\end{thm}

The remark that we made after Theorem \ref{thm-gamma} applies in this case as well: if we are willing to
abandon the uniform convergence requirement, the similar upper
bound as in Theorem \ref{thm-f} can be established for a fixed
target density $\rho$ without an assumption that it is necessarily
symmetric. We state the corresponding theorem below.

\begin{thm}
\label{thm-f2} Assume that Conditions
\ref{conditionv}--\ref{conditionk} hold. Let $\lambda>0,\sigma>0$
and let $\rho$ satisfy \eqref{rhocond}. In addition, suppose that
\begin{equation*} \infint
|t|^{2\beta}|\phi_{f}(t)|^2dt<\infty.
\end{equation*}
Let the estimator $\hat{\rho}_n$ be defined by \eqref{deconvnoise-hatfnh}. Then
\begin{equation*}
\operatorname{MISE}[\hat{\rho}_{n}]\lesssim
(\log n)^{-\beta}.
\end{equation*}
\end{thm}

\section{Lower bound for estimation of $\rho$}
\label{lowbound}

In the previous section we showed that under certain smoothness
assumptions on the class of target densities $\rho,$ the
convergence rate of our estimator $\hat{\rho}_{n}$ is logarithmic.
This convergence rate can be easily understood on an intuitive
level when comparing our problem to a deconvolution problem, see
e.g.\ Section 10.1 of \cite{wasserman} for an introduction to
deconvolution problems. A deconvolution problem consists of
estimation of a density (or a distribution function) of a directly
unobservable random variable $Y$ based on i.i.d.\ copies
$X_1,\ldots,X_n$ of a random variable $X=Y+Z.$ The $X$'s can be
thought of as repetitive measurements of $Y,$ which are corrupted
by an additive measurement error $Z.$ It is well-known that if the
distribution of $Z$ is normal, and if the class of the target
densities is sufficiently large, e.g.\ some H\"older class (see
Definition 1.2 in \cite{tsyb}), the minimax convergence rate will
be logarithmic for both the mean squared error and mean integrated
squared error as measures of risk, see \cite{fan1,fan2}. We will
prove a similar result for a problem of estimation of a L\'evy
density $\rho.$
\begin{thm}
\label{lowbound-thm} Denote by $T$ an arbitrary L\'evy triplet
$(\gamma,\sigma^2,\rho),$ such that
$|\gamma|\leq\Gamma,\sigma\in(0,\Sigma],\lambda\in(0,\Lambda].$
Furthermore, let
\begin{equation}
\label{fcnd} \infint |t|^{2\beta}|\phi_f(t)|^2dt\leq C
\end{equation}
for $\beta\geq 1/2.$ Let $\mathcal{T}$ be a collection of all such
triplets. Then
\begin{equation*}
\inf_{\widetilde{\rho}_n}\sup_{\mathcal{T}}\operatorname{MISE}[\widetilde{\rho}_{n}]\gtrsim (\log n)^{-\beta},
\end{equation*}
where the infimum is taken over all estimators
$\widetilde{\rho}_n$ based on observations $X_1,\ldots,X_n.$
\end{thm}
Using similar techniques, it is expected that lower bounds of the
logarithmic order can be obtained for estimation of
$\gamma,\sigma^2$ and $\lambda$ as well. Such a result is not
surprising e.g.\ for $\sigma^2,$ if one recalls comparable results
from \cite{matias} for estimation of the error variance in the
supersmooth deconvolution problem. Another paper containing
examples of the breakdown of the usual root ${n}$ convergence rate
for estimation of a finite-dimensional parameter is
\cite{ishwaran}. We do not pursue this question any further. We
also notice that the logarithmic lower bounds for estimation of
the components of a characteristic triplet (under a different
observation scheme) were obtained in \cite{belomestny}.

Our estimation procedure for $\rho$ in Section
\ref{deconvnoise-results} relies on the assumption that the random
variable $X$ has a density (the latter is ensured by the condition
$\sigma>0$). If $\sigma=0,$ then an approach of \cite{gug} may be
used for estimation of $\rho.$ For completeness purposes, however,
we will show that the lower bound for the minimax risk in this
case is not logarithmic as in Theorem \ref{lowbound-thm}, but
polynomial.

\begin{thm}
\label{thm-lbnd} Let $\mathcal{T}$ denote a collection of L\'evy
triplets $T=(\gamma,0,\rho),$ such that $|\gamma|\leq \Gamma$ and
$\lambda\in(0,\Lambda].$ Furthermore, let $\phi_f$ satisfy
\eqref{fcnd} for $\beta\geq 1/2.$ Then
\begin{equation*}
\inf_{\widetilde{\rho}_n}\sup_{\mathcal{T}}\operatorname{MISE}[\widetilde{\rho}_{n}]\gtrsim
n^{-2\beta/(2\beta+1)},
\end{equation*}
where the infimum is taken over all estimators
$\widetilde{\rho}_n$ based on observations $X_1,\ldots,X_n.$
\end{thm}

This theorem in essence says that estimation of the L\'evy density
$\rho$ in the case $\sigma=0$ seems to be as difficult as e.g.\
nonparametric density estimation based on i.i.d.\ observations
coming from the target density itself, see e.g.\ \cite{vaart0}.
This result has a parallel in \cite{belomestny}. In absence of the
corresponding upper bound for estimation of $\rho$ nothing can be
said about how sharp the lower bound in Theorem \ref{thm-lbnd} is,
but in any case the polynomial minimax convergence rate seems to
be natural. An upper bound of order $n^{-beta/(2\beta+1)}$ has
been obtained in the compound Poisson model in \cite{comte} for
the mean integrated squared error when estimating $x\rho(x)$ under
the condition that the class of L\'evy densities is a Sobolev
class $\Sigma(\beta,C).$

\section{Proofs}
\label{deconvnoise-proofs}
We first prove the following technical lemma.
\begin{lem}
\label{trlemma}
Let the sets $B_n$ and $B_n^c$ be defined by \eqref{bndef}. Suppose Conditions \ref{conditionh} and \ref{conditionm} hold. Then there exists an integer $n_0,$ such that on the set $B_n^c$ for all $n\geq n_0$ we have
\begin{equation}
\label{tr1}
\max\{\min\{M_n,\log(|\phi_{emp}(t)|)\},-M_n\}=\log(|\phi_{emp}(t)|)
\end{equation}
for $t$ restricted to the interval $[-h^{-1},h^{-1}]$ and for all
$\rho\in W(\beta,L,\Lambda,K),\sigma\in(0,\Sigma]$ and
$|\gamma|\leq\Gamma.$ Furthermore,
\begin{equation}
\label{tr2}
\max\{\min\{M_n,\arg(\phi_{emp}(t))\},-M_n\}=\arg(\phi_{emp}(t))
\end{equation}
for $t$ restricted to the interval $[-h^{-1},h^{-1}]$ and for all
$\rho\in W_{sym}(\beta,L,\Lambda,K),\sigma\in(0,\Sigma]$ and
$|\gamma|\leq\Gamma.$ Here $\arg$ denotes the imaginary part of
the distinguished logarithm of $\phi_{emp}(t),$ i.e.\ a continuous
version of its argument, such that $\arg\phi_{emp}(0)=0.$
\end{lem}
\begin{proof}
Formula \eqref{tr1} can be seen as follows:
\begin{equation}
\label{dlr}
\begin{split}
|\log(|\phi_{emp}(t))||&\leq |\log(|\phi_X(t)|)|+\left|\log\left(\left|\frac{\phi_{emp}(t)}{\phi_X(t)}\right|\right)\right|\\
&\leq |\log(|\phi_X(t)|)|+\left|\frac{\phi_{emp}(t)}{\phi_X(t)}-1\right|+\left|\frac{\phi_{emp}(t)}{\phi_X(t)}-1\right|^2\\
&\leq |\log(|\phi_X(t)|)|+\frac{3}{4}\\
&\leq 2\Lambda+\frac{\Sigma^2}{2h^2}+\frac{3}{4}.
\end{split}
\end{equation}
Here in the third line we used an elementary inequality
$|\log(1+z)-z|\leq |z|^2$ valid for $|z|<1/2$ and the fact
that on the set $B_n^c$ we have
\begin{equation}
\label{logineqhalf}
\left|\left|\frac{\phi_{emp}(t)}{\phi_X(t)}\right|-1\right|\leq\left|\frac{\phi_{emp}(t)}{\phi_X(t)}-1\right|<\frac{1}{2},
\end{equation}
while in the last line we used the bound $|\log|\phi_X(t)||\leq
2\Lambda+{\Sigma^2}/{(2h^2)}.$ The equality \eqref{tr1} now is
immediate from Conditions \ref{conditionh} and \ref{conditionm},
because the upper bound for $|\log(|\phi_{emp}(t)|)|$ grows slower
than $M_n.$ Next we prove \eqref{tr2}. The symmetry of $\rho$
implies that $\phi_{\rho}$ is real-valued and hence
$\arg(\phi_X(t))=0.$ On the set $B_n^c$ we have
$|\arg(\phi_{emp}(t))|\leq 2\pi,$ because the path $\phi_{emp}(t)$
cannot make a turn around zero on this set. This proves
\eqref{tr2}, since $M_n$ diverges to infinity.
\end{proof}

Now we are ready to prove Theorems \ref{thm-sigmatilde}--\ref{thm-f}.

\begin{proof}[Proof of Theorem \ref{thm-sigmatilde}]
Write
\begin{equation*}
\ex[(\tilde{\sigma}_n^2-\sigma^2)^2]=\ex[(\tilde{\sigma}_n^2-\sigma^2)^2 1_{B_n}]+\ex[(\tilde{\sigma}_n^2-\sigma^2)^2 1_{B_n^c}]
=I+II,
\end{equation*}
where the set $B_n$ is defined as in \eqref{bndef}. For $I$ we have
\begin{align*}
I&\lesssim \left(M_n^2\left(\int_{-1/h}^{1/h}|v^h(t)|dt\right)^2+\Sigma^4\right)\operatorname{P}(B_n)\\
&\lesssim \left(M_n^2\left(\int_{-1/h}^{1/h}|v^h(t)|dt\right)^2+\Sigma^4\right)\frac{e^{\Sigma^2/h^2}}{nh^2}\\
&=\left(M_n^2h^4\left(\int_{-1}^{1}|v(t)|dt\right)^2+\Sigma^4\right)\frac{e^{\Sigma^2/h^2}}{nh^2},
\end{align*}
where we used Theorem \ref{phideviation} to see the second line.
Observe that under Conditions \ref{conditionh} and
\ref{conditionm} the last term in the above chain of inequalities
converges to zero faster than $h^{2\beta+6}.$ Now we turn to $II.$
On the set $B_n^c,$ for $n$ large enough, truncation in the definition of
$\tilde{\sigma}_n^2$ becomes unimportant, see Lemma \ref{trlemma},
and we have
\begin{align*}
II&=\ex\left[\left(\int_{-1/h}^{1/h}\log(|\phi_{emp}(t)|)v^h(t)dt-\sigma^2\right)^21_{B_n^c}\right]\\
&=\ex\left[\left(\int_{-1/h}^{1/h}\log\left(\left|\frac{\phi_{emp}(t)}{\phi_X(t)}\right|\right)v^h(t)dt+\int_{-1/h}^{1/h}\log(|\phi_X(t)|)v^h(t)dt-\sigma^2\right)^21_{B_n^c}\right].
\end{align*}
Using this fact, \eqref{sig1} and an elementary inequality $(a+b)^2\leq 2(a^2+b^2),$ we obtain that
\begin{align*}
II & \lesssim \Lambda^2\left(\int_{-1/h}^{1/h}\Re(\phi_f(t))v^h(t)dt\right)^2\\
& + \ex\left[\left(\int_{-1/h}^{1/h}\log\left(\left|\frac{\phi_{emp}(t)}{\phi_X(t)}\right|\right)v^h(t)dt\right)^21_{B_n^c}\right]\\
& = III+IV.
\end{align*}
For $III$ we have
\begin{align*}
III & \lesssim h^{2\beta}\left(\int_{-1/h}^{1/h}t^{\beta}\Re(\phi_f(t))\frac{v^h(t)}{(ht)^{\beta}}dt\right)^2\\
& \lesssim h^{2\beta+6}\left(\infint |t^{\beta}||\Re(\phi_f(t))|dt\right)^2\\
& \lesssim h^{2\beta+6}\left(\infint |t^{\beta}||\phi_f(t)|dt\right)^2\\
& \lesssim h^{2\beta+6}\\
& \lesssim (\log n)^{-\beta-3},
\end{align*}
where in the second line we used Condition \ref{conditionv}, to obtain the third line we used the fact that $|\Re(\phi_f(t))|\leq |\phi_f(t)|+|\phi_f(-t)|,$ while the fourth line follows from Condition \ref{conditionf}. We turn to $IV.$ We have
\begin{align*}
IV & \lesssim \ex\left[\left(\int_{-1/h}^{1/h}\left|\frac{\phi_{emp}(t)}{\phi_X(t)}-1\right|v^h(t)dt\right)^2 1_{B_n^c}\right]\\
& + \ex\left[ \left(\int_{-1/h}^{1/h}\left\{\log\left(\left|\frac{\phi_{emp}(t)}{\phi_X(t)}\right|\right)-\left(\left|\frac{\phi_{emp}(t)}{\phi_X(t)}\right|-1\right)\right\}v^h(t)dt\right)^2 1_{B_n^c} \right]\\
& = V+VI.
\end{align*}
Some further bounding and an application of the Cauchy-Schwarz inequality give
\begin{equation*}
V\lesssim
e^{4\Lambda+\Sigma^2/h^2}\int_{-1/h}^{1/h}(v^h(t))^2dt\ex\left[\int_{-1/h}^{1/h}|\phi_{emp}(t)-\phi_X(t)|^2dt\right].
\end{equation*}
Parseval's identity and Proposition 1.7 of \cite{tsyb} applied to
the sinc kernel then yield
\begin{equation*}
\ex\left[\int_{-1/h}^{1/h}|\phi_{emp}(t)-\phi_X(t)|^2dt\right]=2\pi
\ex\left[\int_{-1/h}^{1/h}(q_n(x)-\ex[q_n(x)])^2dx\right]\lesssim
\frac{1}{nh},
\end{equation*}
whence
\begin{equation*}
V\lesssim e^{\Sigma^2/h^2}h^4\frac{1}{n}.
\end{equation*}
As far as $VI$ is concerned, using \eqref{logineqhalf}, an elementary inequality $|\log(1+z)-z|\leq |z|^2,$ valid for $|z|<1/2,$ and the Cauchy-Schwarz inequality, we obtain that
\begin{align*}
VI & \lesssim \int_{-1/h}^{1/h}(v^h(t))^2dt\ex\left[\int_{-1/h}^{1/h}\left|\frac{\phi_{emp}(t)}{\phi_X(t)}-1\right|^4dt1_{B_n^c}\right]\\
& \leq \frac{1}{4}\int_{-1/h}^{1/h}(v^h(t))^2dt\ex\left[\int_{-1/h}^{1/h}\left|\frac{\phi_{emp}(t)}{\phi_X(t)}-1\right|^2dt\right]\\
& \lesssim e^{\Sigma^2/h^2}\int_{-1/h}^{1/h}(v^h(t))^2dt\ex\left[\int_{-1/h}^{1/h}\left|\phi_{emp}(t)-\phi_X(t)\right|^2dt\right].
\end{align*}
Hence $VI$ can be analysed in the same way as $V.$ From the above bounds on $V$ and $VI$ it also follows that $IV$ is negligible in comparison to $III.$ Combination of all these intermediate results completes the proof of the theorem.
\end{proof}

\begin{proof}[Proof of Theorem \ref{thm-lambdatilde}]
The proof is quite similar to that of Theorem \ref{thm-sigmatilde}. Write
\begin{equation*}
\ex[(\tilde{\lambda}_n-\lambda)^2]=\ex[(\tilde{\lambda}_n-\lambda)^2 1_{B_n}]+\ex[(\tilde{\lambda}_n-\lambda)^2 1_{B_n^c}]=I+II.
\end{equation*}
By an argument similar to that in the proof of Theorem \ref{thm-sigmatilde},
\begin{equation*}
I\lesssim (M_n^2\left(\int_{-1}^{1}|u(t)|dt\right)^2+\Lambda^2)\frac{e^{\Sigma^2/h^2}}{nh^2}.
\end{equation*}
This is negligible compared to $h^{2\beta+2}.$ Now we turn to $II.$\label{IIbound} We have
\begin{align*}
II & =\ex\left[\left(\int_{-1/h}^{1/h}\log(|\phi_{emp}(t)|)u^h(t)dt-\lambda\right)^21_{B_n^c}\right]\\
&=\ex\left[\left(\int_{-1/h}^{1/h}\{\log\left(\left|\frac{\phi_{emp}(t)}{\phi_X(t)}\right|\right)+\log(|\phi_X(t)|)\}u^h(t)dt-\lambda\right)^21_{B_n^c}\right]\\
& \lesssim \Lambda^2\left(\int_{-1/h}^{1/h}\Re(\phi_f(t))u^h(t)dt\right)^2\\
& + \ex\left[\left(\int_{-1/h}^{1/h}\log\left(\left|\frac{\phi_{emp}(t)}{\phi_X(t)}\right|\right)u^h(t)dt\right)^21_{B_n^c}\right]\\
& = III+IV.
\end{align*}
Here in the third line we used \eqref{lam1}. Similar as we did it for $III$ in the proof of Theorem \ref{thm-sigmatilde},
one can check that in this case as well $III\lesssim h^{2\beta+2}.$ As far as $IV$ is
concerned, it is of order $e^{\Sigma^2/h^2}n^{-1},$ which can be seen by exactly the same reasoning as in the proof of Theorem
\ref{thm-sigmatilde}. Combination of these results completes the proof of the theorem, because under Condition \ref{conditionh} the dominating term is $III.$
\end{proof}

\begin{proof}[Proof of Theorem \ref{phideviation}]
By Chebyshev's inequality
\begin{equation*}
\operatorname{P}(B_n)\leq\frac{1}{\delta^2} \ex\left[\left(\sup_{t\in[-h^{-1},h^{-1}]}|\phi_{emp}(t)-\phi_X(t)|\right)^2\right].
\end{equation*}
Thus we need to bound the expectation on the right-hand side. This
will be done via reasoning similar to that on pp.\ 326--327 in
\cite{matias}. For all unexplained terminology and notation used
in the sequel we refer to Chapter 2 of \cite{vaart}. Notice that
\begin{equation*}
\ex\left[\left(\sup_{t\in[-h^{-1},h^{-1}]}|\phi_{emp}(t)-\phi_X(t)|\right)^2\right]=\frac{1}{n}\ex\left[\left(\sup_{t\in[-h^{-1},h^{-1}]}|G_n v_t|\right)^2\right].
\end{equation*}
Here $G_n v_t$ denotes an empirical process defined by
\begin{equation*}
G_n v_t=\frac{1}{\sqrt{n}}\sum_{j=1}^n(v_t(X_j)-\operatorname{E}v_t(X_j)),
\end{equation*}
where the function $v_t:x\mapsto e^{itx}.$ Introduce the functions $v_t^{1}:x\mapsto\cos(tx)$ and $v_t^2:x\mapsto \sin(tx).$ Then
\begin{align*}
\ex\left[\left(\sup_{t\in[-h^{-1},h^{-1}]}|G_n v_t|\right)^2\right]&\lesssim \ex\left[\left(\sup_{t\in[-h^{-1},h^{-1}]}|G_n v_t^1|\right)^2\right]\\
&+\ex\left[\left(\sup_{t\in[-h^{-1},h^{-1}]}|G_n v_t^2|\right)^2\right].
\end{align*}
As it will turn out below, both terms on the right-hand side can be treated in the same manner. Observe that the mean value theorem implies
\begin{equation}
\label{envelope}
|v_t^i(x)-v_s^i(x)|\leq |x||t-s|
\end{equation}
for $i=1,2,$ i.e.\ $v_t^i$ is Lipshitz in $t.$ Theorem 2.7.11 of
\cite{vaart} applies and gives that the bracketing number $N_{[]}$
of the class of functions $\mathbb{F}_n$ (this refers either to
$v_t^1$ or $v_t^2$ for $|t|\leq h^{-1}$) is bounded by the
covering number $N$ of the interval $I_n=[-h^{-1},h^{-1}],$ i.e.\
\begin{equation*}
N_{[]}(2\epsilon\left\|x\right\|_{\mathbb{L}_2(Q)};\mathbb{F}_n;\mathbb{L}_2(Q))\leq N(\epsilon;I_n;|\cdot|).
\end{equation*}
Here $Q$ is any discrete probability measure, such that $\left\|x\right\|_{\mathbb{L}_2(Q)}>0.$ Since
\begin{equation*}
N(\epsilon\left\|x\right\|_{\mathbb{L}_2(Q)};\mathbb{F}_n;\mathbb{L}_2(Q))\leq N_{[]}(2\epsilon\left\|x\right\|_{\mathbb{L}_2(Q)};\mathbb{F}_n;\mathbb{L}_2(Q)),
\end{equation*}
see p.\ 84 in \cite{vaart}, and trivially
\begin{equation*}
N(\epsilon;I_n;|\cdot|)\leq \frac{2}{\epsilon}\frac{1}{h},
\end{equation*}
we obtain that
\begin{equation}\
\label{ent1}
N(\epsilon\left\|x\right\|_{\mathbb{L}_2(Q)};\mathbb{F}_n;\mathbb{L}_2(Q))\leq \frac{2}{\epsilon}\frac{1}{h}.
\end{equation}
Define $J(1,\mathbb{F}_n),$ the entropy of the class $\mathbb{F}_n,$ as
\begin{equation*}
J(1,\mathbb{F}_n)=\sup_Q\int_0^1 \{1+\log
(N(\epsilon\left\|x\right\|_{\mathbb{L}_2(Q)};\mathbb{F}_n;\mathbb{L}_2(Q)))
\}^{1/2}d\epsilon,
\end{equation*}
where the supremum is taken over all discrete probability measures
$Q,$ such that $\left\|x\right\|_{\mathbb{L}_2(Q)}>0.$ Since
$\mathbb{F}_n$ is a measurable class of functions with a
measurable envelope (the latter follows from \eqref{envelope}), by
Theorem 2.14.1 in \cite{vaart} we obtain that
\begin{equation*}
\ex\left[\left(\sup_{t\in[-h^{-1},h^{-1}]}|G_n
v_t^i|\right)^2\right]\lesssim
\left\|x\right\|^2_{\mathbb{L}_2(\operatorname{P})}(J(1,\mathbb{F}_n))^2,
\end{equation*}
where the probability $\operatorname{P}$ refers to
$\operatorname{P}_{\gamma,\sigma^2,\rho}.$ Now notice that
\begin{equation*}
\left\|x\right\|^2_{\mathbb{L}_2(\operatorname{P})}=\ex[(\gamma+Y+\sigma
Z )^2]\lesssim \gamma^2+\ex[Y^2]+\sigma^2,
\end{equation*}
where $Y:=\sum_{j=1}^{N(\lambda)}W_j$ denotes the Poisson sum of
i.i.d.\ random variables $W_j$ with density $f,$ while $Z$ is a
standard normal variable. Under conditions of the theorem the term
\begin{equation*}
\ex[Y^2]=\lambda^2 \left(\infint xf(x)dx\right)^2+\lambda\infint
x^2f(x)dx,
\end{equation*}
is bounded uniformly in $\rho.$ Hence
$\left\|x\right\|^2_{\mathbb{L}_2(\operatorname{P})}$ is also
bounded uniformly in $\rho,\sigma$ and $\gamma.$ Using
\eqref{ent1}, the entropy can be further bounded as
\begin{equation*}
J(1,\mathbb{F}_n)\leq \int_0^1\left\{1+\log \left(\frac{2}{\epsilon}\frac{1}{h}\right)\right\}^{1/2}d\epsilon.
\end{equation*}
Here we implicitly assume that $n$ is large enough, so that we take a square root of a positive number. Working out the integral, it is not difficult to check that $J(1,\mathbb{F}_n)=O(h^{-1}).$ Combination of these results yields the statement of the theorem.
\end{proof}

\begin{proof}[Proof of Theorem \ref{thm-gamma}]
Again, the proof is quite similar to that of Theorem
\ref{thm-sigmatilde}. Write
\begin{equation*}
\ex[(\tilde{\gamma}_n-\gamma)^2]=\ex[(\tilde{\gamma}_n-\gamma)^2
1_{B_n}]+\ex[(\tilde{\gamma}_n-\gamma)^2 1_{B_n^c}]=I+II.
\end{equation*}
For $I$ we have
\begin{equation*}
I\lesssim
\left(M_n^2h^2\left(\int_{-1}^1|w(t)|dt\right)^2+\Gamma^2\right)\operatorname{P}(B_n).
\end{equation*}
Thanks to Theorem \ref{phideviation} the right-hand side converges
to zero as $n\rightarrow\infty.$ Moreover, it is negligible
compared to $h^{2\beta+4}.$ Next we turn to $II.$ By Lemma
\ref{trlemma} on the set $B_n^c$ for $n$ large enough truncation in the definition of
$\tilde{\gamma}_n$ becomes unimportant and we have
\begin{align*}
II & =
\ex\left[\left(\int_{-1/h}^{1/h}\Im(\Log(\phi_{emp}(t)))w^h(t)dt-\gamma\right)^21_{B_n^c}\right]\\
& \lesssim
\Lambda^2\ex\left[\left(\int_{-1/h}^{1/h}\Im(\phi_f(t))w^h(t)dt\right)^2 1_{B_n^c}\right]\\
& +
\ex\left[\left(\int_{-1/h}^{1/h}\Im\left(\Log\left(\frac{\phi_{emp}(t)}{\phi_X(t)}\right)\right)w^h(t)dt\right)^21_{B_n^c}\right]\\
& = III + IV.
\end{align*}
The same reasoning as in Theorem \ref{thm-sigmatilde} shows that
here as well $III$ is of order $h^{2\beta+4}.$ As far as $IV$
is concerned, the inequality $|\Im(z)|\leq |z|$ implies that
\begin{equation*}
IV\lesssim
\ex\left[\left(\int_{-1/h}^{1/h}\left|\Log\left(\frac{\phi_{emp}(t)}{\phi_X(t)}\right)\right|w^h(t)dt\right)^21_{B_n^c}.
\right]
\end{equation*}
Now notice that on the set $B_n^c$ the inequality
\begin{equation}
\label{logineq}
\left|\Log\left(\frac{\phi_{emp}(t)}{\phi_X(t)}\right)-\left(\frac{\phi_{emp}(t)}{\phi_X(t)}-1\right)\right|\leq
\left|\frac{\phi_{emp}(t)}{\phi_X(t)}-1\right|^2
\end{equation}
holds, cf.\ formula (4.8) in \cite{gug}. Therefore
\begin{align*}
IV & \lesssim
\ex\left[\left(\int_{-1/h}^{1/h}\left|\frac{\phi_{emp}(t)}{\phi_X(t)}-1\right|w^h(t)dt\right)^21_{B_n^c}\right]\\
& +
\ex\left[\left(\int_{-1/h}^{1/h}\left|\frac{\phi_{emp}(t)}{\phi_X(t)}-1\right|^2w^h(t)dt\right)^21_{B_n^c}\right].
\end{align*}
Just as for $IV$ in the proof of Theorem \ref{thm-sigmatilde},
one can check that in this case as well $IV$ is negligible in
comparison to $III.$ Combination of these results completes the
proof of the theorem.
\end{proof}

\begin{proof}[Proof of Theorem \ref{thm-gam2}]
The proof follows essentially the same steps as the proof of Theorem \ref{thm-gamma}. The only significant difference is that we have to verify that there exists an integer $n_0,$ such that on the set $B_{nh}^c$ for all $n\geq n_0$ truncation in the definition of $\tilde{\gamma}_n$ is unimportant for an arbitrary $\rho$ satisfying conditions of the theorem, and not necessarily for a symmetric $\rho$ as in Lemma \ref{trlemma}. To see this, first notice that
\begin{align*}
\Im(\Log(\phi_{emp}(t)))&=\Im(\Log(e^{\lambda}e^{\sigma^2 t^2/2}\phi_{emp}(t)))\\
\Im(\Log(\phi_{X}(t)))&=\Im(\Log(e^{\lambda}e^{\sigma^2 t^2/2}\phi_{X}(t)))=\Im(e^{\lambda}\phi_f(t)).
\end{align*}
Let $\psi:{\mathbb{R}}\rightarrow {\mathbb{C}},$ where
\begin{equation*}
\psi(t)=\phi_X(t)e^{\lambda}e^{t^2/2}=e^{\lambda\phi_f(t)}.
\end{equation*}
By the Riemann-Lebesgue theorem $\psi(t)$ converges to $1$ as $|t|\rightarrow\infty$ and hence there exists $t^{*}>0,$
such that
\begin{equation}
\label{cut1}
|\psi(t)-1|<\frac{e^{-\lambda}}{2}, \quad |t|>t^{*}.
\end{equation}
Furthermore, we have
\begin{equation}
\label{cut2} |\psi(t)|\geq e^{-\lambda}, \quad t\in {\mathbb{R}}.
\end{equation}
Since $f$ has a finite second moment, by Theorem 1 on p.\ 182 of
\cite{schwartz} the characteristic function $\phi_f$ is
continuously differentiable. Consequently, so is the exponent
$\psi.$ Therefore the path $\psi:[-t^{*},t^{*}]\rightarrow
{\mathbb{C}}$ is rectifiable, i.e.\ has a finite length. In view
of this fact and \eqref{cut2}, $\psi:[-t^{*},t^{*}]\rightarrow
{\mathbb{C}}$ cannot spiral infinitely many times around zero
(because otherwise it would have an infinite length) and for
$|t|>t^{*}$ it cannot make a turn around zero at all because of
\eqref{cut1}. Since $M_n$ diverges to infinity, it follows that
for every $\omega\in B_{nh}^c$ there exists $n_0(\omega),$ such
that $h_{n_0}^{-1}\geq t^{*}$ and for all $n\geq n_0(\omega)$
\begin{equation}
\label{omega1}
\max\{\min\{M_n,\Im(\Log(\phi_{emp}(t)))\},-M_n\}=\Im(\Log(\phi_{emp}(t))).
\end{equation}
However, it is easy to see that in fact there exist a universal integer $n_0,$ such that \eqref{omega1} holds for all $\omega\in B_{nh}^c$: just notice that for each $\omega$ the number of turns that $\phi_{emp}(t)$ makes around zero is determined by the number of turns $m$ that $\psi(t)$ makes around zero and cannot be greater than $2m,$ say. Consequently, there exists a universal bound $4m\pi$ on $\Im(\Log(\phi_{emp}(t)))$ valid for all $\omega\in B_{nh}^c.$ This concludes the proof of the theorem.
\end{proof}

\begin{proof}[Proof of Theorem \ref{thm-f}]
We have
\begin{align*}
\ex\left[\infint |\hat{\rho}_n(x)-\rho(x)|^2dx\right]&=\ex\left[\infint |\hat{\rho}_n(x)-\rho(x)|^2dx1_{B_n}\right]\\
&+\ex\left[\infint |\hat{\rho}_n(x)-\rho(x)|^2dx1_{B_n^c}\right]\\
&=I+II,
\end{align*}
where $B_n$ and $B_n^c$ are defined by \eqref{bndef}. Notice that
\begin{equation*}
\infint |\hat{\rho}_n(x)-\rho(x)|^2dx\lesssim \infint |\hat{\rho}_n(x)|^2dx+\infint |\rho(x)|^2dx.
\end{equation*}
By Parseval's identity and Condition \ref{conditionf}
\begin{equation*}
\infint |\rho(x)|^2dx\lesssim 1.
\end{equation*}
For the Fourier transform of $\hat{\rho}_n$ we have
\begin{equation*}
|\phi_{\hat{\rho}_n}(t)|\lesssim M_n1_{[-h^{-1},h^{-1}]}(t).
\end{equation*}
Hence by Parseval's identity
\begin{equation*}
\infint |\hat{\rho}_n(x)|^2dx\lesssim
M_n^2\frac{1}{h}.
\end{equation*}
Using this and Theorem \ref{phideviation}, we get that
\begin{equation*}
I\lesssim
\left\{M_n^2\frac{1}{h}+1\right\}\frac{e^{\Sigma^2/h^2}}{nh^2}.
\end{equation*}
Under Conditions \ref{conditionh} and \ref{conditionm} the latter is negligible in comparison to $h^{2\beta}.$ Now we turn to $II.$ By Parseval's identity
\begin{align*}
II&=\frac{1}{2\pi}\ex\left[\infint |\phi_{\hat{\rho}_n}(t)-\phi_{\rho}(t)|^2dt1_{B_n^c}\right]\\
&=\frac{1}{2\pi}\ex\left[\int_{-1/h}^{1/h}|\phi_{\hat{\rho}_n}(t)-\phi_{\rho}(t)|^2dt1_{B_n^c}\right]+\frac{1}{2\pi}\int_{\mathbb{R}\setminus(-h^{-1},h^{-1})}
|\phi_{\rho}(t)|^2dt\operatorname{P}(B_n^c)\\
&=III+IV.
\end{align*}
For $IV$ we have
\begin{align*}
IV&\leq \int_{\mathbb{R}\setminus(-h^{-1},h^{-1})}
|\phi_{\rho}(t)|^2dt=\lambda^2\int_{\mathbb{R}\setminus(-h^{-1},h^{-1})}
|t^{2\beta}|\frac{|\phi_{\rho}(t)|^2}{|t^{2\beta}|}dt\\
&\leq \lambda^2h^{2\beta}\infint |t^{2\beta}||\phi_f(t)|^2dt\\
&\leq C\Lambda^2h^{2\beta},
\end{align*}
where the last inequality follows from the definition of the class
$W_{sym}^*(\beta,L,C,\Lambda,K).$ Next we turn to $III.$ With
\eqref{tr1} and \eqref{tr2} we have that
\begin{equation*}
III=\frac{1}{2\pi}\ex\left[\int_{-1/h}^{1/h}|\phi_{\rho_n}(t)-\phi_{\rho}(t)|^2dt 1_{B_n^c}\right]
\end{equation*}
for all $n$ large enough. Consequently,
\begin{align*}
III & \lesssim \ex\left[\left({\tilde{\sigma}_n^2}-{\sigma^2}\right)^2\int_{-1/h}^{1/h}t^4dt1_{B_n^c}\right]\\
& + \ex\left[\int_{-1/h}^{1/h}\left|\Log(\phi_{emp}(t))-\Log(\phi_X(t))\right|^2 1_{B_n^c}\right]\\
& + \ex\left[(\tilde{\gamma}_n-\gamma)^2\int_{-1/h}^{1/h}t^2dt1_{B_n^c}\right]\\
& + \ex\left[(\tilde{\lambda}_n-\lambda)^2\int_{-1/h}^{1/h}dt1_{B_n^c}\right]\\
& =
IV+V+VI+VII.
\end{align*}
For $IV$ we have by Theorem \ref{thm-sigmatilde} that
\begin{equation*}
IV \lesssim \frac{1}{h^5}\ex\left[\left({\tilde{\sigma}_n^2}-{\sigma^2}\right)^21_{B_n^c}\right]=O(h^{2\beta+1}).
\end{equation*}
As far as $V$ is concerned, by the inequality \eqref{logineq}
\begin{equation*}
V\lesssim \ex\left[ \int_{-1/h}^{1/h} \left|\frac{\phi_{emp}(t)}{\phi_X(t)}-1\right|^2dt 1_{B_n^c} \right]+\ex\left[ \int_{-1/h}^{1/h} \left|\frac{\phi_{emp}(t)}{\phi_X(t)}-1\right|^4dt 1_{B_n^c} \right].
\end{equation*}
The right-hand side can be analysed similar to $V$ in the
proof of Theorem \ref{thm-sigmatilde} and in fact it is negligible in
comparison to $h^{2\beta}.$ Furthermore, by Theorem \ref{thm-gamma} $VI$
is of order $h^{2\beta+1}.$ Also $VII$ is of order $h^{2\beta+1}$ by Theorem \ref{thm-lambdatilde}. Combination of all the intermediate
results completes the proof of the theorem.
\end{proof}

\begin{proof}[Proof of Theorem \ref{thm-f2}]
The proof uses the same type of arguments as that of Theorem \ref{thm-f}. The only essential difference is to show that there exists $n_0,$ such that on the set $B_{nh}^c$ for all $n\geq n_0$ we have $\hat{\rho}_n(x)=\rho_n(x).$ We therefore consider in detail only this part of the proof. For $\arg(\phi_{emp}(t))$ the corresponding argument was already given in the proof of Theorem \ref{thm-gam2}. Thus we only have to prove that
\begin{equation*}
\max\{\min\{M_n,\log(|\phi_{emp}(t))\},-M_n\}1_{B_{nh}^c}=\log(|\phi_{emp}(t)|)1_{B_{nh}^c}.
\end{equation*}
The latter can be shown by exactly the same arguments that were used in the proof of \eqref{tr1} in Lemma \ref{trlemma}.
\end{proof}

\begin{proof}[Proof of Theorem \ref{lowbound-thm}]
The proof makes use of some of the ideas found in
\cite{tsyb1,fan1}. Consider two L\'evy triplets
$T_1=(0,\sigma^2,\rho_1)$ and $T_2=(0,\sigma^2,\rho_2),$ where
$\rho_i(x)=\lambda f_i(x),i=1,2$ and $\lambda<\Lambda.$ Let
\begin{equation*}
f_1(x)=\frac{1}{2}(r_1(x)+r_2(x)),
\end{equation*}
where the probability densities $r_1$ and $r_2$ are defined via their characteristic functions,
\begin{equation*}
r_1(x)=\frac{1}{2\pi}\infint
e^{-itx}\frac{1}{(1+t^2/\beta_1^2)^{(\beta_2+1)/2}}dt; \quad
r_2(x)=\frac{1}{2\pi}\infint
e^{-itx}e^{-\alpha_1|t|^{\alpha_2}}dt.
\end{equation*}
With a proper selection of $\beta_1,\beta_2,\alpha_1$ and
$\alpha_2$ one can achieve that $f_1$ satisfies \eqref{fcnd} with
a constant $C/4$ (instead of $C$). We also assume that
$1<\alpha_2<2.$ Notice that $r_1$ is a bilateral gamma density,
while $r_2$ is a stable density. To define $f_2,$ we perturb $f_1$
as follows:
\begin{equation*}
f_2(x)=f_1(x)+\delta_n^{\beta-1/2}H(x/\delta_n),
\end{equation*}
where $\delta_n\rightarrow 0$ as $n\rightarrow\infty,$ and the
function $H$ satisfies the following conditions:
\begin{enumerate}
\item $\infint |t|^{2\beta}|\phi_H(t)|^2dt\leq C/4;$
\item $\infint H(x)dx=0;$
\item $\int_{-\infty}^0 H(x)dx\neq 0;$
\item $\phi_H(t)=0$ for $t$ outside $[1,2];$
\item $\phi_H(t)$ is twice continuously differentiable.
\end{enumerate}
To see why such a function exists, see e.g.\ p.~1268 in
\cite{fan1}. It is also obvious, that there are many functions $H$
with an appropriate tail behaviour, such that $f_2(x)\geq 0$ for
all $x\in\mathbb{R},$ at least for small enough $\delta_n.$ With
such an $H$ and small enough $\delta_n,$ the function $f_2$ will
be a probability density satisfying \eqref{fcnd}. Notice that
\begin{equation}
\label{distance} \infint (\rho_2(x)-\rho_1(x))^2dx\asymp
\delta_n^{2\beta}.
\end{equation}
Here the symbol $\asymp$ means `asymptotically of the same order'.
Denote by $q_i$ a density of a random variable $X$ corresponding
to a triplet $T_i, i=1,2.$ The statement of the theorem will
follow from \eqref{distance} and Lemma 8 of \cite{tsyb1}, if we
prove that the $\chi^2$-divergence (see p.\ 72 in \cite{tsyb} for
a definition) between $q_2$ and $q_1$ satisfies
\begin{equation}
\label{chi-square} n\chi^2(q_2,q_1)=n\infint
\frac{(q_2(x)-q_1(x))^2}{q_1(x)}dx\leq c,
\end{equation}
where a positive constant $c<1$ is independent of $n.$

Let $g_i$ be a density of a Poisson sum $Y$ conditional on the
fact that its number of summands $N(\lambda)>0.$ Here the index
$i$ refers to a triplet $T_i,i=1,2.$ Since
\begin{equation}
\label{ychf}
\phi_Y(t)=e^{-\lambda}+(1-e^{-\lambda})\frac{1}{e^{\lambda}-1}\left(e^{\lambda\phi_{f_i}(t)}-1\right),
\end{equation}
it follows that
\begin{equation*}
\phi_{g_i}(t)=\frac{1}{e^{\lambda}-1}\left(e^{\lambda\phi_{f_i}(t)}-1\right).
\end{equation*}
We also have
\begin{equation}\label{g-expr}
g_i(x)=\sum_{n=1}^{\infty}f_i^{\ast
n}(x)P(N(\lambda)=n|N(\lambda)>0).
\end{equation}
From \eqref{ychf} we obtain
\begin{equation*}
q_1(x)\geq (1-e^{-\lambda})\phi_{0,\sigma^2}\ast g_1(x),
\end{equation*}
where $\phi_{0,\sigma^2}$ denotes a normal density with mean zero
and variance $\sigma^2.$ Moreover, by Lemma 2 of \cite{matias},
there exists a large enough constant $A,$ such that the right-hand
side of the above display is not less than
$(1-e^{-\lambda})g_1(|x|+A).$ Hence
\begin{equation*}
n\chi^2(q_2,q_1)\lesssim n\infint
\frac{(q_2(x)-q_1(x))^2}{g_1(|x|+A)}dx\lesssim n\infint
\frac{(q_2(x)-q_1(x))^2}{f_1(|x|+A)}dx,
\end{equation*}
where the last inequality follows from \eqref{g-expr}. Splitting
the integration region into two parts, we then get that
\begin{align*}
n\chi^2(q_2,q_1)&\lesssim n\int_{|x|\leq A} {(q_2(x)-q_1(x))^2}dx+n\int_{|x|>A} x^4{(q_2(x)-q_1(x))^2}dx\\
&=I+II.
\end{align*}
Here we used the fact that $f_1(x)$ behaves as $|x|^{-1-\alpha_2}$
at plus and minus infinity, see e.g.\ formula (14.37) in
\cite{sato}, and that $1<\alpha_2<2.$ Since
\begin{equation*}
\delta_n^{\beta-1/2}\infint
e^{itx}H(x/\delta_n)dx=\delta_n^{\beta+1/2}\phi_H(\delta_n t),
\end{equation*}
by Parseval's identity it holds that
\begin{align*}
I&\leq n\frac{1}{2\pi}\infint
|\phi_{q_2}(t)-\phi_{q_1}(t)|^2dt\\
&=n\frac{(1-e^{-\lambda})^2}{2\pi}\infint
|\phi_{g_2}(t)-\phi_{g_1}(t)|^2e^{-\sigma^2t^2}dt\\
&=n\frac{(1-e^{-\lambda})^2}{(e^{\lambda}-1)^2}\frac{1}{2\pi}\infint
|e^{\lambda\phi_{f_2}(t)}-e^{\lambda\phi_{f_1}(t)}|^2e^{-\sigma^2t^2}dt\\
& \lesssim n\infint
|\phi_{f_2}(t)-\phi_{f_1}(t)|^2e^{-\sigma^2t^2}dt,
\end{align*}
where the last inequality follows from the mean-value theorem
applied to the function $e^x$ and the fact that
$|\lambda\phi_{f_i}(t)|\leq\lambda.$ By definition of $f_1$ and
$f_2$ we then get that
\begin{align*}
I&\lesssim n\delta_n^{2\beta+1}\infint |\phi_H(\delta_n
t)|^2e^{-\sigma^2t^2}dt\\
&=n\delta_n^{2\beta}\infint
|\phi_H(s)|^2e^{-\sigma^2s^2/\delta_n^2}ds\\
&=O\left(n\delta_n^{2\beta}e^{-\sigma^2/\delta_n^2}\right).
\end{align*}
The choice $\delta_n\asymp (\log n)^{-1/2}$ with small enough
constant will now imply that $I\rightarrow 0$ as
$n\rightarrow\infty.$ Next we turn to $II.$ By Parseval's identity
\begin{equation*}
II\leq n\frac{1}{2\pi}\infint
|(\phi_{q_2}(t)-\phi_{q_1}(t))''|^2dt.
\end{equation*}
Here we use the fact that even though $\phi_{f_1}$ and
$\phi_{f_2}$ are not twice differentiable at zero, the difference
$\phi_{q_2}(t)-\phi_{q_1}(t)$ still is, because $\phi_H$ is
identically zero outside the interval $[1,2].$ By exactly the same
type of arguments as we used for $I,$ one can show that
$II\rightarrow 0$ as $n\rightarrow\infty,$ provided
$\delta_n\asymp (\log n)^{-1/2}.$ Hence \eqref{chi-square} is
satisfied and the statement of the theorem follows.
\end{proof}

\begin{proof}[Proof of Theorem \ref{thm-lbnd}]
The proof is similar to the proof of Theorem \ref{lowbound-thm}.
Let $\rho_1(x)=\lambda f_1(x)$ with $f_1$ as in the proof of
Theorem \ref{lowbound-thm}. Consider a perturbation of $\rho_1,$
say $\rho_2(x)=\lambda f_2(x),$ where $f_2$ is defined as in
Theorem 4.1. Assume that the function $H$ in the definition of
$f_2$ has a compact support on $[-1,1]$ and that it satisfies
Conditions 1--3 in the proof of Theorem \ref{lowbound-thm}. This
implies that $f_2(x)\geq 0$ for $\delta_n$ small enough. Therefore
$\rho_2$ is a L\'evy density satisfying \eqref{fcnd}, provided
$\delta_n$ is small enough. Denote by $\mathbb{P}_{1n}$ and
$\mathbb{P}_{2n}$ the laws of a L\'evy process $X=(X)_{t\geq 0}$
restricted to the time interval $[0,n]$ and corresponding to the
characteristic triplets $T_1=(0,0,\rho_1)$ and $T_2=(0,0,\rho_2),$
respectively. Notice that
\begin{equation}
\label{cineq}
\inf_{\widetilde{\rho}_n}\sup_{\mathcal{T}}\ex\left[\infint
(\widetilde{\rho}_n(x)-\rho(x))^2dx\right]\geq
\inf_{{\rho}_n}\sup_{\mathcal{T}}\ex\left[\infint({\rho}_n(x)-\rho(x))^2dx\right],
\end{equation}
where ${\rho}_n$ denotes an arbitrary estimator based on a
continuous record of observations of $X$ over $[0,n].$ Let
$K(P,Q)$ denote the Kullback-Leibler divergence between the
probability measures $P$ and $Q,$
\begin{equation*}
K(P,Q)=
\begin{cases}
\int\log\frac{dP}{dQ}dP & \text{if $P\ll Q$,}\\
+\infty & \text{if otherwise},
\end{cases}
\end{equation*}
see Definition 2.5 in \cite{tsyb}. In view of \eqref{distance},
the result will follow from formula \eqref{cineq} above, the
arguments of Section 2.2 of \cite{tsyb} combined with Theorem 2.2
(iii) of \cite{tsyb}, provided the Kullback-Leibler divergence
$K(\mathbb{P}_{2n},\mathbb{P}_{1n})$ between the measures
$\mathbb{P}_{2n}$ and $\mathbb{P}_{1n}$ remains bounded for all
$n$ by a constant independent of $n.$ The Kullback-Leibler
divergence between $\mathbb{P}_{2n}$ and $\mathbb{P}_{1n}$ can be
easily computed via Theorem A.1 of \cite{cont}, which in our case
gives that $K(\mathbb{P}_{2n},\mathbb{P}_{1n})=n
K(\rho_2,\rho_1),$ because both $\rho_1$ and $\rho_2$ have the
same total mass. Let $\chi^2(\rho_2,\rho_1)$ denote the
$\chi^2$-divergence between the densities $\rho_2$ and $\rho_1.$
It is not difficult to see that
$K(\rho_2,\rho_1)\leq\chi^2(\rho_2,\rho_1),$ cf.\ formula (2.20)
in \cite{tsyb}. It follows that in order to prove the theorem, it
suffices to show that $\chi^2({\rho}_2,{\rho}_1)=O(n^{-1}).$ By
definition of $\rho_1,\rho_2,H$ and a change of the integration
variable we have that
\begin{equation}
\label{chidist} \chi^2(\rho_2,\rho_1)\lesssim
\delta_n^{2\beta+1}\int_{-1}^{1}\frac{(H(u))^2}{f_1(\delta_n u
)}du.
\end{equation}
The dominated
convergence theorem implies that the right-hand side of the above
equation is of order $\delta_n^{2\beta+1}.$ Taking $\delta_n\asymp
n^{-1/(2\beta+1)}$ gives that \eqref{chidist} is of order
$n^{-1}.$ This yields the statement of the theorem.
\end{proof}
\medskip

\noindent {\bf Acknowledgments.} The author would like to thank
Bert van Es and Peter Spreij for discussions on various parts of
the draft version of the paper. Part of the research was done
while the author was at Korteweg-de Vries Institute for
Mathematics in Amsterdam. The research at Korteweg-de Vries
Institute for Mathematics was financially supported by the
Nederlandse Organisatie voor Wetenschappelijk Onderzoek (NWO).

\end{document}